\documentclass[reqno,a4paper,12pt]{amsart}
\usepackage{amsmath,amsthm,amsfonts,amssymb, mathrsfs}
\usepackage{verbatim, graphicx, ifthen, enumitem,color}
\usepackage[T1]{fontenc}
\usepackage [applemac] {inputenc}

\allowdisplaybreaks

\let\oldmarginpar\marginpar
\renewcommand\marginpar[1]{\-\oldmarginpar[\raggedleft\footnotesize #1]%
{\raggedright\footnotesize #1}}
\newtheorem{theorem}{Theorem}
\newtheorem{lemma}{Lemma}

\theoremstyle{definition}

\theoremstyle{remark}
\newtheorem{remark}{Remark}

\newcommand{\dif}{\mathrm{d}}
\newcommand{\e}{\mathrm{e}}

\newcommand{\sinc}{\mathrm{sinc} }

\newcommand{\Z}{\mathbb{Z}}

\newcommand{\abs}[1]{|#1|}
\newcommand{\Abs}[1]{\left|#1\right|}

\newcommand{\N}{\mathbb{N}}
\newcommand{\R}{\mathbb{R}}

\newcommand{\C}{\mathbb{C}}

\def \Z {\mathbb{Z}}
\def \C {\mathbb{C}}
\def \N {\mathbb{N}}
\def \R {\mathbb{R}}

\def \P {\mathbb{P}}
\def \E {\mathbb{E}}

\renewcommand{\Im}{\operatorname{Im}}

\begin{document}

\title[Gap probabilities]{Gap probabilities for  the cardinal sine}

\author{Jorge Antezana, Jeremiah Buckley, Jordi Marzo \\and Jan-Fredrik Olsen}

\address{Universidad Nacional de La Plata, Departamento de Matem\'atica, Esq. 50 y 115 s/n, Facultad de Ciencias La Plata (1900), Buenos Aires, Argentina}
\email{antezana@mate.unlp.edu.ar}

\address{Departament de matem\`atica aplicada i an\`alisi, Universitat de Barcelona, Gran Via 585, 08007, Barcelona, Spain}
\email{jerry.buckley07@gmail.com}

\address{Departament de matem\`atica aplicada i an\`alisi, Universitat de Barcelona, Gran Via 585, 08007, Barcelona, Spain}
\email{jmarzo@ub.edu}

\address{Centre for Mathematical Sciences, Lund University, P.O. Box 118, SE-221 00 Lund, Sweden}
\email{janfreol@maths.lth.se}


\section*{PRELIMINARY VERSION}

\date{\today}
\begin{abstract}
We study the zero set of random analytic functions generated by  a sum of the cardinal sine functions that form an orthogonal basis for the Paley-Wiener space. As a model case, we consider  real-valued  Gaussian coefficients. It is shown that the  asymptotic probability that there is no zero  in a bounded interval decays exponentially as a function of the length.
\end{abstract}

\maketitle


\section{Introduction}
A simple point process in $\R$ is a random integer-valued positive Radon measure on $\R$   that almost surely assigns at most measure 1 to singletons.
Simple point processes can be identified with random discrete subsets of $\R.$ In this paper, we   study `gap probabilities' of the simple point
process in $\R$ given by the zeros of the random function
\begin{equation*}
	 f(z)=\sum_{n\in\Z}a_n \frac{\sin \pi (z-n)}{\pi (z-n)},
\end{equation*}
where $a_n$ are i.i.d.\ random variables with zero mean and unit variance. Kolmogorov's inequality shows that this sum is almost surely pointwise
convergent. In fact, since
\begin{equation*}
	 \sum_{n\in\Z}\left|\frac{\sin \pi (z-n)}{\pi (z-n)}\right|^2
\end{equation*}
converges uniformly on compact subsets of the plane, this series almost surely defines an entire function. If we take $a_n$ to be Gaussian random
variables then $f$ is a Gaussian analytic function (GAF) (see \cite[Lemma 2.2.3]{HKPV} for details). 

We shall be chiefly concerned with the functions
given by taking $a_n$ to be real Gaussian random variables. We denote by $n_f$ the counting measure on the set of zeros of $f$. These functions are an
example of a stationary symmetric GAF and Feldheim   \cite{F} has shown that the density of zeros is given by
\begin{equation}\label{density}
	\E[n_f(z)]=S(y)m(x,y)+\frac{1}{2\sqrt{3}} \mu(x),
\end{equation}
where $z=x+iy$, $m$ denotes the planar Lebesgue measure, $\mu$ is the singular measure with respect to $m$ supported on $\R$ and identical to Lebesgue measure there, and
\begin{equation*}
	 S\Big(\frac{y}{2\pi}\Big)=\pi\left|\frac{d}{dy}\left(\frac{\cosh y-\frac{\sinh y}{y}}{\sqrt{\sinh^2y-y^2}}\right)\right|.
\end{equation*}
(Here $S$ is defined only for $y\neq0$, in fact the atom appearing in \eqref{density} is the distributional derivative at $0$.) We observe that since
$S(y)=O(y)$ as $y$ approaches zero there are almost surely zeros on the real line, but that they are   sparse close by. Moreover the zero set is on
average uniformly distributed on the real line. We are interested in the `gap probability', that is the probability that there are no zeros in a large
interval on the line. Our result is the following asymptotic estimate.
\begin{theorem}                                                        \label{gap_paley_wiener}
      Let $f$ be the symmetric GAF given by the almost surely convergent series
$$\sum_{n \in \mathbb{Z}} a_n \frac{\sin \pi (z-n)}{\pi (z-n)},$$
    where $a_n$ are i.i.d.\ real Gaussian variables with mean $0$ and variance $1.$
    Then, there exist constants $c,C>0$ such that for all $r\ge 1$,
$$e^{-c r}\le \mathbb{P}(\# (Z(f)\cap (-r,r))=0)\le e^{-C r}.$$
\end{theorem}
\begin{remark}
If instead of considering intervals we consider the rectangle $D_r=(-r,r)\times(-a,a)$ for some fixed $a>0$, then we compute a similar exponential decay
for $\mathbb{P}(\# (Z(f)\cap D_r)=0)$
\end{remark} 
 	\begin{remark}
	Suppose that the $a_n$ are i.i.d.\ Rademacher distributed. I.e.,  each $a_n$ is equal to either $-1$ or $1$ with
	equal probability. Since $f(n)  = a_n$ for $n \in \N$, it follows that if not all $a_n$ for $\abs{n} \leq N$ are of equal
	sign, then by the mean value theorem, $f$ has to have a zero in $(-N,N)$. As is shown below, the remaining two
	choices of the $a_n$ for $\abs{n}\leq N$ each yield an $f$ without zeroes in $(-N,N)$. Clearly,
	this gives a probability of $2(1/2)^{2N}$ for $f$ to be without zeroes there.
 	\end{remark}
 \begin{remark}
	The Cauchy distribution is given by the density
	\begin{equation*}
		p(x) =  \frac{1}{\pi}  \frac{1}{x^2 + 1}.
	\end{equation*}
	Whereas the Rademacher distribution is in some sense a simplified Gaussian, the Cauchy distribution is very different: It has neither 
	an expectation, nor a standard deviation. If we suppose that the $a_n$ are i.i.d.\ Cauchy distributed,  it is not hard to see that with probability one the sum $\sum a_n/n$ diverges, whence   the related random function diverges everywhere.
 \end{remark}
 
The main motivation for our work comes from the `hole theorem' proved by Sodin and Tsirelson in \cite{ST} for point-processes uniformly
distributed in the plane. The authors consider the GAF defined by
\begin{equation*}
	 F(z)=\sum_{n=0}^\infty a_n \frac{z^n}{\sqrt{n!}},
\end{equation*}
where $a_n$ are i.i.d.\ standard complex normal variables. For this function, the density of zeros is proportional to the planar Lebesgue measure and the
authors compute the asymptotic probability that there are no zeros in a disc of radius $r$ to be $e^{-cr^4}$, where $c>0$. The analogy with the GAF we
consider becomes clear if we note that $(\frac{z^n}{\sqrt{n!}})_{n=0}^\infty$ constitutes an orthonormal basis for the Bargmann-Fock space  
\begin{equation*}
	 \mathcal{F} =\{f\in H(\mathbb{C}):\|f\|_{\mathcal{F}}^2=\int_\C|f(z)|^2 e^{-2|z|^2}\frac{dm(z)}{\pi}<+\infty\},
\end{equation*}
where $m$ is the planar Lebesgue measure. We are replacing these functions with the sinc functions, which constitute a basis for the Paley-Wiener space.
An important caveat is that, though $f$ is constructed from an orthonormal basis, $f$ is almost surely not in the Paley-Wiener space, since the sequence
of coefficients $a_n$ is almost surely not in $\ell^2(\Z)$. It is not hard to see however that $f$ belongs almost surely to the Cartwright class.



\section{Proof of Theorem \ref{gap_paley_wiener}}


\subsection{Upper bound}

  We want to compute the probability of an event that contains the event of not having any zeroes on $(-N,N),$
  for $N\in \N.$
  One such event is that the values $f(n)$ have the same sign for $\abs{n}\leq N$. The probability of this event is
\begin{equation*}
    \P\Big( a_n > 0 \; \text{for} \; \abs{n} \leq N \quad \text{or} \quad a_n < 0 \; \text{for} \abs{n} \leq N\Big)
    = 2 (1/2)^{2N+1} = \e^{-CN},
\end{equation*}
    for some constant $C>0.$

\begin{remark}
  The same upper bound holds when $a_n$ are i.i.d.\ random variables with $0<\P(a_n>0)<1$ for which the random function
$\sum_{n \in \mathbb{Z}} a_n \sinc (x-n)$ converges.
\end{remark}


\subsection{Lower bound}


To compute the lower hole probability, we use the following scheme. First, we introduce the deterministic function
\begin{equation*}
     f_0(x) = \sum_{n = -2N}^{2N} \sinc (x-n).
\end{equation*}
We show in Lemma \ref{lemma_non_zero_guy} that it has no zeroes on $(-N,N)$, and we find an explicit lower bound on $(-N,N)$ for it. This lower bound does not depend on $N$. Second, we consider the functions
\begin{equation*}
	f_1(x)  =  \sum_{n = -2N}^{2N} (a_n - 1) \sinc (x-n) \qquad \text{and} \qquad f_2(x) = \sum_{\abs{n} > 2N} a_n  \sinc (x-n),
\end{equation*}
which induce the splitting
\begin{equation*}
    f = f_0 + f_1  +  f_2.
\end{equation*}
We show that for all $x\in [-N,N]$ we have $|f_1 (x)|\le \epsilon$ with probability
at least $\e^{-cN}$ for large $N$ and some constant $c>0.$ Moreover, we show that
$$\P\left(\sup_{x\in [-N,N]}|f_2(x)|\le \epsilon\right)$$ is larger than, say, $1/2$ for big enough $N$. As the events on $f_1$ and $f_2$ are clearly independent, the lower bound now follows by choosing $\epsilon$ small enough.



We turn to the first part of the proof.
\begin{lemma}                                               \label{lemma_non_zero_guy}
  Given $N\in \N$ and
\begin{equation}                                            \label{non_zero_guy}
f_0(x)=\sum_{n = -2N}^{2N} \sinc (x-n)=\sin \pi x \sum_{n = -2N}^{2N} \frac{(-1)^n}{\pi(x-n)}.
\end{equation}
  Then, there exists a constant $C>0$ such that, for $N$ big enough,
$$1-\frac{C}{N}\le \inf_{|x|\le N} f_0(x)\le \sup_{|x|\le N} f_0(x)\le 1+\frac{C}{N}.$$
\end{lemma}

\begin{proof}
Let $R = R(N)$ be the rectangle of   length $4N + 1$ and height $4N$, centered at $x=1/2$. By the residue theorem, it holds that
\begin{equation*}
    \frac{1}{2\pi  i} \oint_{R} \frac{d z}{(z - x) \sin \pi z} = \frac{1}{\pi} \sum_{n=-2N+1}^{N}
\frac{(-1)^n}{n - x} + \frac{1}{\sin \pi x}.
\end{equation*}
Observe that if we shift around the terms,   this yields
\begin{equation*}
    \sin \pi x \sum_{n=-2N+1}^{2N}  \frac{(-1)^n }{x-n} =  1  +     \frac{\sin\pi x }{2\pi i}
\oint_{R} \frac{d z}{(x - z ) \sin \pi z}.
\end{equation*}
    Now, given $-N\le x \le N,$ it is easy to bound this last integral by $C/N.$
\end{proof}


\begin{remark} The same bound holds for all points $z$ in a strip with fixed height $[-N,N]\times [-C,C]$ for some $C>0.$
We observe though, that the function $f_0(z)$ in (\ref{non_zero_guy}) is close to zero around $\Im z = \log N$. Indeed,
it is smaller than $\e^{-cN}$ there, for some $c>0$ independent of $N$.
\end{remark}


\subsubsection{The middle terms}
Let $\epsilon>0$ be given, and consider $N \in \N$ to be fixed. We look at  the function
\begin{equation*}
    f_1(x) = \sum_{n=-2N}^{2N} (a_n-1) \sinc(x-n) = \frac{\sin \pi x  }{\pi} \sum_{n=-2N}^{2N} (a_n-1) \frac{(-1)^n}{x-n}.
\end{equation*}
To simplify the expression, we set $b_n = (a_n-1) (-1)^n.$
  We want to compute  a lower bound for the probability that,   for   $x\in [-N,N]$,
$$|f_1(x)|\le \epsilon.$$
 Since, for $|n|\le N$, we have $f_1(n)=a_n-1$,   the condition
\begin{equation*}
	 |a_n-1|\le \epsilon \quad \text{for} \quad |n|\le N 
\end{equation*}
is necessary.

  Define $B_n=b_{-2N}+\dots b_n$ for $|n|\le 2N$ with $B_{-2N-1}=0$, and suppose that $x \notin \Z$. With this,  summation by parts yields
\begin{equation}                                \label{summation}
\sum_{-2N}^{2N}\frac{b_n}{x-n}=-\sum_{-2N}^{2N}  \frac{B_n}{(x-n)(x-n-1)}+\frac{B_{2N}}{x-2N-1}.
\end{equation}
We now claim that under the event
\begin{equation}                            \label{the event}
	E = \Big\{ \abs{b_n} \leq \epsilon, \; \abs{B_n} \leq \epsilon  \quad \text{for} \quad \abs{n} \leq 2N \Big\},
\end{equation}
  we have $|f(x)-f_0(x)|\le \epsilon$ for $|x|\le N,$ with a bound
  independent of $N$. Indeed,
  the second summand at the right hand side of (\ref{summation}) converges almost surely to zero, because
$$\left| \frac{B_{2N}}{x-2N-1} \right| \le \frac{\epsilon}{N}.$$
  Suppose that $x\in (k,k+1)$ and split the first sum in (\ref{summation}) as
$$\sum_{ \substack{ n=-2N \\ n\neq k-1,k,k+1} }^{2N}  \frac{B_n}{(x-n)(x-n-1)}+
\sum_{n=k-1}^{k+1}\frac{B_n}{(x-n)(x-n-1)}.$$
  Then
$$\left| \sum_{
\substack{ n=-2N \\ n\neq k-1,k,k+1}
}^{2N} \frac{B_n}{(x-n)(x-n-1)} \right|\le  \sum_{n\ge k+2}\frac{\epsilon}{(k+1-n)^2}+ \sum_{n\le k-2}\frac{\epsilon}{(k-n)^2}\lesssim \epsilon.$$
  For the remaining terms, the function $\sin \pi x$ comes into play. E.g., suppose that $|x-k|\le 1/2$, then
$$\left| \sin \pi x \frac{B_k}{(x-k)(x-k-1)}\right| \lesssim \frac{\epsilon}{|x-k-1|} \left| \frac{\sin \pi (x-k)}{\pi(x-k)}\right| \lesssim \epsilon.$$
  The remaining terms are treated in exactly the same way.

  What remains is to compute the probability of the event $E$ defined by  (\ref{the event}). We recall that the $b_n$ were all defined in terms of
the real and independent Gaussian variables $a_n$. So the event $E$ above defines a set
$$V=\left\{ (t_{-2N},\dots t_{2N})\in \R^{4N+1} : |t_n|\le \epsilon,\; |\sum_{-2N}^n t_n|\le \epsilon, \; |n|\le 2N \right\}$$
in terms of the values of the $a_n$. Hence,
\begin{equation*}
    \P(E) = c^{2N} \int \cdots \int_V \e^{-(t_{-2N}^2 + \cdots t_{2N}^2)/2} \dif t_{-2N}  \cdots \dif t_{2N}.
\end{equation*}
Here, $c$ is the normalising constant of the one dimensional Gaussian. Since $\abs{a_n - 1} < \epsilon$, it follows that
\begin{equation*}
    \P(E) \geq c^{2N} \e^{-N(1 + \epsilon)^2}  \int \cdots \int_V  \dif t_{-2N}  \cdots \dif t_{2N}=c^{2N} \e^{-N(1 + \epsilon)^2} \mbox{Vol}(V).
\end{equation*}
We now seek a lower bound for this euclidean $(4N+1)$-volume.

To simplify notation, we pose this problem as follows. For real variables $x_1, \ldots, x_N$,
we wish to compute the euclidean volume of the solid $V_N$ defined by $\abs{x_i} \leq \epsilon$ for $i=1,\dots , N$ and
\begin{align*}
    &\abs{x_1 + x_2} \leq \epsilon, \\
    &\abs{x_1 + x_2 + x_3} \leq \epsilon \\
    & \qquad \qquad \vdots \\
    &\abs{x_1 + x_2 + \cdots + x_N} \leq \epsilon.
\end{align*}
One way to do this is as follows. Write $y_N = x_1 + \cdots + x_{N-1}$, then
\begin{equation*}
    \text{Vol}(V_N) = \int \cdots \int_{V_{N-1}}
\left(\int_{\text{max}\{-\epsilon, - \epsilon - y_N \}}^{\text{min}\{\epsilon,  \epsilon - y_N \}} \dif x_N \right) \dif x_1 \cdots \dif x_{N-1}.
\end{equation*}
This is illustrated in Figure \ref{figur}.
\begin{figure}
       \includegraphics{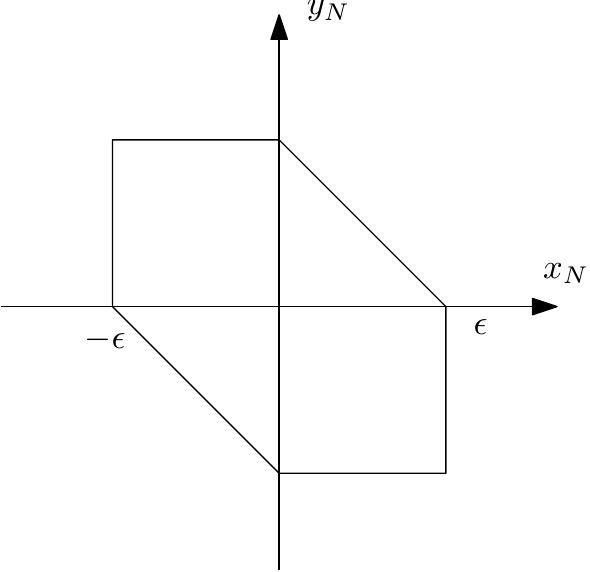}
       \caption{Illustration of the solid $V_N$.\label{figur}}
\end{figure}
Clearly, whenever $y_N < 0$, the upper limit is $\epsilon$, and whenever $y_N >0$, the lower limit is $-\epsilon$. Hence,
\begin{align*}
    \text{Vol}(V_N)  &\geq
    \int \cdots \int_{V_{N-1} \cap \{ y_N < 0 \}} \left( \int_0^\epsilon \dif x_N \right) \dif x_1 \cdots \dif x_{N-1}
    \\ & \qquad \qquad +
    \int \cdots \int_{V_{N-1} \cap \{ y_N > 0 \}} \left( \int_{-\epsilon}^0 \dif x_N \right) \dif x_1 \cdots \dif x_{N-1}
    =\epsilon \text{Vol}(V_{N-1}).
\end{align*}
Iterating  this, we get
\begin{equation*}
    \text{Vol}(V_N) \geq \epsilon^N.
\end{equation*}
In conclusion,
\begin{equation*}
    \P(E) \geq \e^{-c N},
\end{equation*}
  which concludes this part of the proof.


\subsubsection{The tail}

   We now turn to the tail term
\begin{equation*}
    f_2(x) = \sum_{\abs{n} > 2N} a_n \sinc(x-n).
\end{equation*}
Clearly, we need only consider the terms for which $n$ is positive. Set $c_n = (-1)^n a_n$. We factor out the sine factor as above, and
apply summation by parts, to get
\begin{equation} \label{the expression}
     \sum_{n > 2N}^L \frac{c_n}{x-n} = \sum_{2N+1}^L C_n \frac{-1}{(x-n) (x- n-1)}
     + \frac{C_L}{x -L-1}.
\end{equation}
where
\begin{equation*}
     C_n = c_{2N+1} + \cdots +c_n,\;\; C_{2N}=0.
\end{equation*}
We want to take the limit as $L \rightarrow \infty$. It is easy to see that the last term almost surely tends to zero. Indeed, $C_L$ is a sum of
independent normal variables with mean $0$ and variance $1$, and therefore is itself normal with mean $0$ and variance $L-2N.$
Moreover, since
\begin{equation*}
    \Abs{\frac{C_L}{x-L-1}} \lesssim \Abs{\frac{C_L}{N-L}},
\end{equation*}
and the random variable inside of the absolute values on the right-hand side has variance $L-2N$, it follows
that limit is almost surely equal to $0$, whence we are allowed to let $L\rightarrow \infty$ in \eqref{the expression}.

We  prove the following. With a positive probability, we have
\begin{equation*}
    \Abs{ \sum_{2N+1}^L C_n \frac{1}{(x-n) (x- n-1)} } \leq \epsilon.
\end{equation*}
As $n^2\simeq |(x-n)(x-n-1)|$ for $|x|\le N$ and $n>2N$, it is enough to consider the expression
\begin{equation*} 
    \sum_{2N+1}^L \frac{ \abs{C_n}}{n^2}.
\end{equation*}
The absolute value of a Gaussian random variable has the folded-normal distribution. In particular, if $X \sim N(0,\sigma^2)$, then
\begin{equation*}
    \E(\abs{X})  = \sigma \sqrt{ \frac{2}{\pi}}.
\end{equation*}
Since, in our case, $\sigma^2 = n-2N$, this yields
\begin{equation*}
    \E\left( \sum_{2N+1}^L \frac{ \abs{C_n}}{n^2} \right)
    \lesssim \sum_{2N+1}^L  \frac{\sqrt{n-2N}}{n^2} \lesssim \sum_{1}^\infty  \frac{1}{(n+2N)^{3/2}}\lesssim \frac{1}{\sqrt{N}}.
\end{equation*}
Finally, by Chebyshev's inequality,
$$\P\left(\sum_{2N+1}^L \frac{ \abs{C_n}}{n^2} \leq \epsilon\right) \geq 1 - \frac{\E(Y)}{\epsilon} \geq 1 -
\frac{C}{\epsilon \sqrt{N}}.$$


 \section{Acknowledgements}
This work was done as part of the research program ``Complex Analysis and Spectral Problems'' 2010/2011 at the Centre de Recerca
Matem\`atica (CRM), Bellaterra, Barcelona.
We are grateful to Mikhail Sodin and Joaquim Ortega-Cerd\`a for enlightening discussions on the subject matter of this paper.


\end{document}